\documentclass[reqno]{amsart}
\usepackage{amsmath,amsthm,amscd,amssymb,amsfonts, amsbsy}
\usepackage{latexsym}
\usepackage{color}
\usepackage{enumerate}
\usepackage{mathrsfs}
\usepackage{pxfonts}
\usepackage{verbatim}

\theoremstyle{plain}
\newtheorem{theorem}[equation]{Theorem}
\newtheorem{lemma}[equation]{Lemma}

\theoremstyle{definition}

\theoremstyle{remark}

\newtheorem{acknowledgment}[equation]{Acknowledgment}

\newcommand{\dv}{\operatorname{div}}

\newcommand*{\tran}{^{\mkern-1.5mu\mathsf{T}}}

\numberwithin{equation}{section}

\newcommand{\bR}{\mathbb{R}}

\newcommand\sL{\mathscr{L}}

\providecommand{\ip}[1]{\langle#1\rangle}

\providecommand{\set}[1]{\{#1\}}

\providecommand{\abs}[1]{\lvert#1\rvert}
\providecommand{\Abs}[1]{\left\lvert#1\right\rvert}

\providecommand{\norm}[1]{\lVert#1\rVert}

\providecommand{\tri}[1]{\lvert\!\lvert\!\lvert#1\rvert\!\rvert\!\rvert}

\renewcommand{\vec}[1]{\boldsymbol{#1}}

\DeclareMathOperator*{\esssup}{ess\,sup}

\newcommand\rV{\mathring{V}}
\newcommand\rW{\mathring{W}}

\newcommand{\BMO}{\mathrm{BMO}}

\begin{document}
\title[Fundamental matrices]
{Fundamental solutions for second order parabolic systems with drift terms}

\author[H. Dong]{Hongjie Dong}
\address[H. Dong]{Division of Applied Mathematics, Brown University,
182 George Street, Providence, RI 02912, United States of America}
\email{hdong@brown.edu}
\thanks{Hongjie Dong is partially supported by the National Science Foundation under agreement DMS-1600593}

\author[S. Kim]{Seick Kim}
\address[S. Kim]{Department of Mathematics, Yonsei University, 50 Yonsei-ro, Seodaemun-gu, Seoul 03722, Republic of Korea}
\email{kimseick@yonsei.ac.kr}
\thanks{Seick Kim is partially supported by National Research Foundation of Korea under agreement NRF-2016R1D1A1B03931680.}



\begin{abstract}
We construct fundamental solutions of second-order parabolic systems of divergence form with bounded and measurable leading coefficients and divergence free first-order coefficients in the class of $\BMO^{-1}_x$,
under the assumption that weak solutions of the system satisfy a certain local boundedness estimate. We also establish Gaussian upper bounds for such fundamental solutions under the same conditions.
\end{abstract}

\maketitle

\section{Introduction} 
In this paper, we study fundamental solutions (or fundamental matrices) of second-order parabolic systems of divergence form
\[
\sum_{j=1}^m \sL_{ij}u^j:=
u^i_t- \sum_{j=1}^m \sum_{\alpha,\beta=1}^n D_{\alpha}(A^{\alpha\beta}_{ij} D_{\beta}u^j)+ \sum_{j=1}^m \sum_{\alpha=1}^n B^{\alpha}_{ij} D_{\alpha} u^j + \sum_{i,j=1}^m C_{ij} u^j,\quad i=1,\ldots, m.
\]
By using matrix notation and adopting the usual summation convention over repeated indices, we write the above system as
\begin{equation}				\label{master_eq}
\sL\vec u :=\vec u_t-D_\alpha (\vec A^{\alpha\beta}\,D_\beta \vec u)+ \vec B^\alpha D_{\alpha} \vec u + \vec C \vec u,
\end{equation}
where $\vec A^{\alpha\beta}=\vec A^{\alpha\beta}(t,x)$, $\vec B^\alpha=\vec B^\alpha(t,x)$, and $\vec C=\vec C(t,x)$ are $m \times m$ matrix valued functions defined on $\bR\times \bR^n=\bR^{n+1}$ and $\vec u=(u^1,\ldots,u^m)\tran$ is a column vector valued function on $\bR^{n+1}$.

We assume that the principal coefficients $\vec A^{\alpha\beta}$ satisfy the following parabolicity and boundedness condition: there are constants $0<\lambda, \Lambda<\infty$ such that
\begin{equation}				\label{cond_a1}
\lambda \sum_{i=1}^m \sum_{\alpha=1}^n \, \abs{\xi^i_\alpha}^2 \le  A^{\alpha\beta}_{ij} \xi^i_\alpha \xi^j_\beta,
\end{equation}
\begin{equation}				\label{cond_a2}
\sum_{i,j=1}^m \sum_{\alpha, \beta=1}^n \,\abs{A^{\alpha\beta}_{ij}}^2 \le \Lambda^2.
\end{equation}
Note that we do not impose any symmetry condition on $A^{\alpha\beta}$.
We also assume that $\vec B^\alpha$ is symmetric and divergence free and that $\vec C$ is nonnegative definite; that is
\begin{equation}				\label{cond_b1}
B^\alpha_{ij}=B^\alpha_{ji},\quad D_\alpha B^\alpha_{ij}=0,
\end{equation}
\begin{equation}				\label{cond_c}
C_{ij} \xi^i \xi^j \ge 0, \quad \forall (\xi^1,\ldots, \xi^m) \in \bR^m.
\end{equation}
Finally, we assume that $\vec B^\alpha \in L^\infty_t(\BMO^{-1}_x)$; that is there are $m\times m$ matrix valued function $\vec \Phi^{\alpha\beta}$ in $\bR^{n+1}$ and a constant $0<\Theta <\infty$ such that
\begin{equation}				\label{cond_b2}
B^\alpha_{ij}=D_\beta \Phi^{\alpha\beta}_{ij},\quad \sum_{i,j=1}^m \sum_{\alpha,\beta=1}^n \sup_{t\in \bR}\,\norm{\Phi_{ij}^{\alpha\beta}(t,\cdot)}_{\BMO(\bR^n)}^2 \le \Theta^2.
\end{equation}
The system of the form \eqref{master_eq} is relevant for applications to incompressible flows. See, for instance, \cite{Z04, SSSZ}.

By a fundamental solution for the system \eqref{master_eq}, we mean an $m\times m$ matrix valued function
$\vec \Gamma(t,x,s,y)$ ($x,y \in\bR^n$ and $t,s\in\bR$) which satisfies the following:
\begin{align*}
\sL_{t,x} \,\vec \Gamma(t,x,s,y)=0 &\quad\text{ in }\; (s,\infty) \times \bR^n, \\
\vec \Gamma (t,x,s,y)=\delta_{y}(x) \vec I &\quad\text{ on }\; \set{t=s} \times \bR^n,
\end{align*}
where $\delta_{y}(\cdot)$ is a Dirac delta function and $\vec I$ is the $m\times m$ identity matrix; see Theorem~\ref{thm1} for more precise definition.
Since $\vec B^\alpha$ is divergence free, the adjoint operator $\sL^\ast$ is given as follows:
\[
\sL^\ast \vec u :=-\vec u_t-D_\alpha ({}^\ast\vec A^{\alpha\beta} D_\beta \vec u)- \vec B^\alpha D_{\alpha} \vec u + \vec C \vec u,
\]
where ${}^\ast\vec A^{\alpha\beta}= (\vec A^{\beta\alpha})\tran$ (i.e., ${}^\ast A^{\alpha\beta}_{ij}=A^{\beta\alpha}_{ji}$).
Note that the coefficients ${}^\ast\vec A^{\alpha\beta}$ satisfy the same parabolicity and boundedness conditions \eqref{cond_a1} and \eqref{cond_a2}.

The goal of this article is to show that if $\sL$ and $\sL^\ast$ both satisfy the local boundedness property with constant $N_0$ (see Section~\ref{sec:lb} below), then there exists a fundamental solution $\vec \Gamma(t,x,s,y)$ of the system \eqref{master_eq} which satisfies the following Gaussian bound: there exist constants $C=C(n,m,\lambda, \Lambda, \Theta, N_0)$ and $\kappa=\kappa(n,m, \lambda, \Lambda, \Theta)>0$ such that for all $t, s \in \bR$ satisfying $s<t$ and $x, y \in \bR^n$, we have
\[
\abs{\vec\Gamma(t,x,s,y)} \le \frac{C}{(t-s)^{n/2}} \exp \left\{- \frac{\kappa \abs{x-y}^2}{t-s} \right\}.
\]

A few historical remarks are in order.
Fundamental solutions of parabolic equations of divergence form with bounded
measurable coefficients have been studied by many authors.
The first significant step in this direction was made in 1957 by Nash \cite{Nash},
who established certain estimates of the fundamental solutions in proving
local H\"older continuity of weak solutions.
In 1967, Aronson \cite{Aronson} proved Gaussian upper and lower bounds for the fundamental solutions by using the parabolic Harnack inequality of Moser \cite{Moser}.
In 1986, Fabes and Stroock \cite{FS} showed that the idea of Nash could be
used to establish Aronson's Gaussian bounds, which consequently gave
a new proof of Moser's parabolic Harnack inequality.
In 2008, the authors and Cho \cite{CDK08} considered parabolic systems \eqref{master_eq} without lower-order terms (i.e., $\vec B^\alpha = \vec C =0$) and constructed the fundamental solutions and obtained Gaussian upper bounds under the assumption that weak solutions to the system and its adjoint system are locally H\"older continuous.
For the fundamental solutions of parabolic equations with measurable coefficients in nondivergence form, a paper by Escauriaza \cite{Escauriaza} is notable.

In writing this article, we are very much motivated by very recent papers by Qian and Xi \cite{QianXi16, QianXi17}.
They considered parabolic equations with divergence-free drift terms and established upper and lower Gaussian bounds.
Earlier in 2012, Seregin et al. \cite{SSSZ} studied scalar parabolic equations $\partial_t u - \dv(A \nabla u)=0$ and established the Moser's Harnack inequality under the assumption that $A=a+d$, where $a$ is symmetric, $d$ is skew symmetric, and satisfies
\[
\lambda \vec I \le a \le \Lambda \vec I,\quad d \in L^\infty_t(\BMO_x).
\]
It is more or less straightforward to check that the above scalar equations are covered by the parabolic system introduced at the beginning.
In the spirit of Fabes and Stroock \cite{FS}, Moser's Harnack inequality should be equivalent to having the two-sided Gaussian bounds for the fundamental solution.
As a matter of fact, it is proved in \cite{HofKim2} that local boundedness property (which is implied by Moser's Harnack inequality) implies Gaussian bounds for the fundamental solution for parabolic systems.
However, it was not clear that the fundamental solutions for the aforementioned scalar equation considered in \cite{SSSZ} enjoy Gaussian bounds.
In \cite{QianXi16}, Qian and Xi resolved this question by using a clever inequality involving Hardy norm; see Proposition~3.2 in \cite{QianXi16}.
By adopting the inequality by Qian and Xi to the systems setting, we are able to extend the main result in \cite{CDK08} to parabolic systems with drift terms satisfying the aforementioned conditions, which are the natural extension of the conditions imposed in \cite{SSSZ, QianXi16}. 
Because of the well-known embedding $L^n\hookrightarrow \BMO^{-1}$ (see, for instance, \cite{KT2001}), our result also extends Theorem 2 of \cite{QianXi17}. We refer the reader to \cite{Z04, S06, QianXi16, QianXi17} and the references therein for other previous results in this direction.

The organization of the paper is as follows.
In Section~\ref{main}, we introduce some notation and preliminary lemmas, and then state our main result, Theorem~\ref{thm1}.
Section~\ref{sec3} is devoted to the proof of the main theorem.

\section{Preliminaries and main results} \label{main}
We use the same notation as used in \cite{CDK08}.
For reader's convenience, we reproduce the most frequently used notation here.
We refer the reader to \cite{CDK08} for more details.
\subsection{Basic Notation}
We use $X=(t,x)$ to denote a point in $\bR^{n+1}=\bR\times \bR^n$.
We define the parabolic distance between the points $X=(t,x)$ and $Y=(s,y)$ in $\bR^{n+1}$ as
\[
\abs{X-Y}_p:=\max(\sqrt{\abs{t-s}},\abs{x-y}).
\]
We use the following notation for basic cylinders in $\bR^{n+1}$:
\begin{align*}
Q^-_r(X)&=(t-r^2,t)\times B_r(x),\\
Q^+_r(X)&=(t,t+r^2)\times B_r(x),\\
Q_r(X)&=(t-r^2,t+r^2)\times B_r(x).
\end{align*}
In the rest of this subsection, we shall denote by $Q$ the cylinder $(t_0,t_1)\times\Omega$.
We denote by $W^{1,0}_2(Q)$ the Hilbert space with the inner product
\[
\ip{u,v}_{W^{1,0}_2(Q)}:=\int_Q uv+\sum_{\alpha=1}^n \int_Q D_\alpha u D_\alpha v
\]
and by $W^{1,1}_2(Q)$ the Hilbert space with the inner product
\[
\ip{u,v}_{W^{1,1}_2(Q)}:=\int_Q uv+\sum_{\alpha=1}^n \int_Q D_\alpha u D_\alpha v
+\int_Q u_t v_t.
\]
We denote by $\rW^{1,0}_2(Q)$ and $\rW^{1,1}_2(Q)$ the closure of $C^\infty_{c}([t_0,t_1]\times \Omega)$ in the Hilbert spaces $W^{1,0}_2(Q)$ and $W^{1,1}_2(Q)$, respectively.
We define $V_2(Q)$ as the Banach space consisting of all elements
of $W^{1,0}_2(Q)$ having a finite norm
\[
\norm{u}_{V_2(Q)}=\tri{u}_{Q}:= \left(\norm{D u}_{L^2(Q)}^2
+\esssup\limits_{t_0 \le t\le t_1}\,\norm{u(t,\cdot)}_{L^2(\Omega)}^2\right)^{1/2}.
\]
The space $V^{1,0}_2(Q)$ is obtained by completing the set $W^{1,1}_2(Q)$ in the norm of $V_2(Q)$.
We define $\rV_2(Q):=V_2(Q)\cap \rW^{1,0}_2(Q)$ and
$\rV^{1,0}_2(Q)=V^{1,0}_2(Q)\cap \rW^{1,0}_2(Q)$.
We recall the following well known embedding theorem (see e.g., \cite[\S II.3]{LSU}):
\begin{equation} \label{eqn:2.2}
\norm{u}_{L^{2+\frac{4}{n}}(Q)} \le C(n) \tri{u}_{Q}
\quad\forall u\in \rV_2(Q).
\end{equation}
\subsection{Energy inequality}			
Due to the assumptions \eqref{cond_b1} and \eqref{cond_c}, the following energy inequality is available for the operator $\sL$ and its adjoint $\sL^\ast$.
\begin{lemma}				
Let $Q=(t_0 ,t_1) \times \Omega$ and $\vec u \in \rV^{1,0}_2(Q)$ be a weak solution of
\[
\sL \vec u = \vec f\;\text{ in }\; Q,\quad \vec u(x,t_0)=\vec \psi(x)\;\text{ on }\;\Omega,
\]
where $\vec \psi \in L^2(\Omega)$ and $\vec f \in L^{(2n+4)/(n+4)}(Q)$.
Then $\vec u$ satisfies the energy inequality
\begin{equation}						\label{energy_ineq}
\tri{\vec u}_Q \le C \left(\norm{\vec \psi}_{L^2(\Omega)}+ \norm{\vec f}_{L^{(2n+4)/(n+4)}(Q)} \right),
\end{equation}
where $C=C(n, \lambda, \Lambda)$.
A similar statement is true for a corresponding adjoint problem.
\end{lemma}
\begin{proof}
Note that assumption \eqref{cond_b1} implies
\[
\int_\Omega B^\alpha_{ij} D_\alpha u^j u^i \,dx = \int_\Omega \frac12 B^{\alpha}_{ij} D_\alpha (u^i u^j) \,dx =0
\]
and assumption \eqref{cond_c} implies
\[
\int_\Omega C_{ij} u^j u^i \,dx \ge 0.
\]
Then, testing the equation with $\vec u$ itself and using \eqref{eqn:2.2}, we obtain \eqref{energy_ineq} as usual.
\end{proof}

\subsection{Local boundedness property} \label{sec:lb}
We shall say that the operator $\sL$ (resp. $\sL^\ast$) satisfies the \emph{local boundedness property} for weak solutions if there exists a constant $N_0$ such that
\begin{equation}				\label{local_bdd}
\norm{\vec u}_{L^\infty(\frac12 Q)} \le N_0\left\{ \left(\fint_{Q} \abs{\vec u}^2\,dxdt\right)^{\frac12} + r^2 \norm{\vec f}_{L^\infty(Q)} \right\}
\end{equation}
whenever $\vec u \in V_2(Q)$ is a weak solution of $\sL \vec u=\vec f$ (resp. $\sL^\ast \vec u=\vec f$) in $Q=Q_r^{-}(X_0)$ (resp. $Q=Q_r^+(X_0)$) and $\frac12 Q= Q_{r/2}^{-}(X_0)$ (resp. $\frac12 Q= Q_{r/2}^{+}(X_0)$).

\subsection{Main result}
We now state our main theorems.
\begin{theorem} \label{thm1}
Let the coefficients of the operator $\sL$ satisfy the conditions \eqref{cond_a1} -- \eqref{cond_b2}.
Assume that operators $\sL$ and $\sL^\ast$ both satisfy the local boundedness property \eqref{local_bdd}.
Then, there exists a unique Green's matrix $\vec \Gamma(X,Y)=\vec\Gamma(t,x,s,y)$ on $\bR^{n+1}\times\bR^{n+1}$ which satisfies $\vec \Gamma(t,x,s,y)\equiv 0$ for $t<s$
and has the property that $\vec\Gamma(X,\cdot)$ is locally integrable in $\bR^{n+1}$ for all $X\in\bR^{n+1}$ and that for all $\vec f\in C^\infty_c(\bR^{n+1})^m$, the function $\vec u$ given by
\[
\vec u(X):=\int_{\bR^{n+1}} \vec \Gamma(X,Y)\vec f(Y)\,dY
\]
is a weak solution in $\rV^{1,0}_2(\bR^{n+1})^m$ of $\sL\vec u=\vec f$.
Also, for all $\vec g\in L^2(\bR^n)^m$, the function $\vec u(t,x)$ given by
\[
\vec u(t,x):=\int_{\bR^n} \vec \Gamma(t,x,s,y)\vec g(y)\,dy
\]
is the unique weak solution in $\rV^{1,0}_2((s,\infty)\times\bR^n)^m$ of the Cauchy problem
\[
\left\{\begin{array}{l l}
\sL \vec u=0\\
\vec u(s,\cdot)= \vec g.\end{array}\right.
\]
Moreover, we have for all $t>s$ and $x,y\in\bR^n$,
\begin{equation}				\label{gaussian}
\abs{\vec\Gamma(t,x,s,y)} \le \frac{C}{(t-s)^{\frac{n}{2}}} \exp \left\{- \frac{\kappa \abs{x-y}^2}{t-s} \right\},
\end{equation}
where $C=C(n,m,\lambda,\Lambda, \Theta, N_0)$ and $\kappa=\kappa(n, m, \lambda,\Lambda, \Theta)>0$ are constants.
\end{theorem}

\section{Proof of Theorem \ref{thm1}}				\label{sec3}
\subsection{Averaged fundamental solution}  
We closely follow the steps used in \cite{CDK08} with appropriate modification.
Let $Y=(s,y)\in \bR^{n+1}$ and $1\le k \le m$ be fixed.
For each $\epsilon>0$, fix $s_0\in (-\infty,s-\epsilon^2)$ and consider the problem
\[
\left\{\begin{array}{l l}
\sL \vec u=\frac{1}{\abs{Q^{-}_\epsilon}}1_{Q^-_\epsilon(Y)} \vec e_k\\
\vec u(s_0,\cdot)= 0,\end{array}\right.
\]
where $\vec e_k$ is the $k$-th unit vector.
By using the energy inequality \eqref{energy_ineq} and following \cite[Chapter III]{LSU}, we find that the above problem has a unique weak solution $\vec v_\epsilon=\vec v_{\epsilon;Y, k}$ in $\rV^{1,0}_2((s_0,\infty)\times\bR^n)$.
Moreover, by the uniqueness, we find that $\vec v_\epsilon$ does not depend
on the particular choice of $s_0$ and we may extend $\vec v_\epsilon$ to
the entire $\bR^{n+1}$ by setting
\[
\vec v_\epsilon\equiv 0\quad\text{on}\quad(-\infty,s-\epsilon^2)\times \bR^n.
\]
Then, by \eqref{energy_ineq} we have
\begin{equation}					\label{eq17.59tu}
\tri{\vec v_\epsilon}_{\bR^{n+1}}\le C\abs{Q_\epsilon^-(Y)}^{-\frac{n}{2n+4}} \le C \epsilon^{-\frac{n}{2}}.
\end{equation}
Next, for each $\vec f\in C^\infty_c(\bR^{n+1})^m$,
let us fix $t_0$ such that $\vec{f}\equiv 0$ on $[t_0,\infty)\times \bR^n$.
We consider the backward problem
\[
\left\{\begin{array}{l l}
\sL^\ast \vec u=\vec f\\
\vec u(t_0,\cdot)= 0.
\end{array}\right.
\]
Again, we obtain a unique weak solution $\vec u$ in $\rV^{1,0}_2((-\infty,t_0)\times\bR^n)$ and we may extend $\vec u$ to the entire $\bR^{n+1}$ by setting
$\vec u\equiv 0$ on $(t_0,\infty)\times\bR^n$.
Then, by the energy inequality \eqref{energy_ineq}, we have
\begin{equation}					\label{eq2.00}
\tri{\vec u}_{\bR^{n+1}} \le C \norm{\vec f}_{L^{2(n+2)/(n+4)}(\bR^{n+1})}
\end{equation}
and similar to \cite[Lemma~3.1]{CDK08}, we have
\begin{equation}				\label{eq2.17}
\int_{\bR^{n+1}} \vec v_\epsilon\cdot\vec f =\fint_{Q^-_\epsilon(Y)} u^k.
\end{equation}

Now, we assume that $\vec f$ is supported in $Q^+_R(X_0)$.
By the local boundedness property \eqref{local_bdd} combined with \eqref{eq2.00} and \eqref{eqn:2.2}, we have
\begin{equation} \label{eq2.19}
\norm{\vec u}_{L^\infty(Q_{R/2}^+(X_0))} \le C R^{2} \norm{\vec f}_{L^\infty(Q_R^+(X_0))}.
\end{equation}
If $Q^-_\epsilon(Y)\subset Q^+_{R/2}(X_0)$, then \eqref{eq2.17} together with
\eqref{eq2.19} yields
\[
\Abs{\int_{Q^+_R(X_0)}\vec v_\epsilon \cdot \vec f \,} \le \fint_{Q^-_\epsilon(Y)}\abs{\vec u}\le CR^{2} \norm{\vec f}_{L^{\infty}(Q^+_R(X_0))}.
\]
By duality, it follows that if $Q^-_\epsilon(Y)\subset Q^+_{R/2}(X_0)$, then
\begin{equation} \label{eq11.09}
\norm{\vec v_\epsilon}_{L^1(Q^+_R(X_0))}\le CR^{2}.
\end{equation}

Finally, we define the \textit{averaged fundamental solution} $\vec \Gamma^\epsilon(\cdot,Y)$ for $\sL$ by setting
\[
\Gamma^\epsilon_{jk}(\cdot,Y)=v_\epsilon^j=v^j_{\epsilon;Y,k}.
\]

\begin{lemma} \label{lem3.07}
Let $X=(t,x)$, $Y=(s,y)$, and assume $X\neq Y$.
Then
\begin{equation} \label{eq11.16}
\abs{\vec \Gamma^\epsilon(X,Y)}\le C \abs{X-Y}_p^{-n}, \quad \forall  \epsilon\le \tfrac13 \abs{X-Y}_p,
\end{equation}
where $C=C(n,m,\lambda, \Lambda, \Theta_0, N_0)$.
\end{lemma}
\begin{proof}
Denote $d=\abs{X-Y}_p$ and let $X_0=(s-4d^2,y)$, $r=d/3$, and $R=20r$.
It is easy to see that
\[
Q^{-}_\epsilon(Y)\subset Q^{+}_{R/2}(X_0),\quad Q^{-}_r(X)\subset Q^{+}_{R}(X_0),
\quad Q^{-}_\epsilon(Y)\cap Q^{-}_r(X)=\emptyset.
\]
Since $\vec v_\epsilon=\vec v_{\epsilon; Y, k}$ is a weak solution of $\sL \vec u=0$ in $Q^-_r(X)$, by the local boundedness property \eqref{local_bdd} and the standard argument (see \cite[pp. 80--82]{Giaq93}), we have
\[
\abs{\vec v_\epsilon(X)} \le C N_0 r^{-(n+2)}\norm{\vec v_\epsilon}_{L^1(Q_r^{-}(X))}.
\]
Therefore, by \eqref{eq11.09}, we have $\abs{\vec v_\epsilon(X)}\le Cr^{-n}$, which implies \eqref{eq11.16}.
\end{proof}

\subsection{Construction of the fundamental matrix}	
Recall that $\vec v_\epsilon \in \rV^{1,0}_2(\bR^{n+1})$ satisfies
\begin{equation}				\label{eq1858sat}
\sL \vec v_\epsilon=\frac{1}{\abs{Q^{-}_\epsilon}}1_{Q^-_\epsilon(Y)} \vec e_k.
\end{equation}
For $\epsilon < \rho < R<\infty$, let $\eta: \bR^{n+1} \to \bR$ be a smooth nonnegative function such that
\begin{equation}				\label{eq_eta}
\eta \equiv 0\;\text{ on }\; Q_{\rho}(Y),\quad \eta \equiv 1\; \text{ on }\;Q_{R}(Y)^c,\quad \abs{D\eta}^2 + \abs{D^2\eta}+ \abs{\eta_t}\le \tfrac{12}{(R-\rho)^2}.
\end{equation}
By testing \eqref{eq1858sat} with $\eta^2 \vec v_\rho$ and using assumption \eqref{cond_b1}, we have
\begin{multline*}
0 = \int_{\bR^n} \frac12 (\eta^2\abs{\vec v_\epsilon}^2)_t -\int_{\bR^n} \eta \eta_t \abs{\vec v_\epsilon}^2+\int_{\bR^n} \eta^2 A^{\alpha\beta}_{ij} D_\beta v_\epsilon^j D_\alpha v_\epsilon^i \\
+ \int_{\bR^n} 2 \eta  A^{\alpha\beta}_{ij} D_\beta v_\epsilon^j D_\alpha \eta v_\epsilon^i -\int_{\bR^n} \eta D_\alpha \eta B^\alpha_{ij} v_\epsilon^i v_\epsilon^j + \int_{\bR^n} \eta^2 C_{ij} v_\epsilon^i v_\epsilon^j.
\end{multline*}
Then by using \eqref{cond_a1}, \eqref{cond_a2}, and \eqref{cond_c}, we get
\begin{multline*}
\int_{\bR^n} \frac12 (\eta^2\abs{\vec v_\epsilon}^2)_t + \lambda \int_{\bR^n} \eta^2 \abs{D \vec v_\epsilon}^2 \\
\le  \int_{\bR^n} \eta \abs{\eta_t}\, \abs{\vec v_\epsilon}^2 + 2\Lambda\int_{\bR^n} \eta \abs{D \vec v_\epsilon}\,\abs{D\eta}\,\abs{\vec v_\epsilon}+  \int_{\bR^n}  B^\alpha_{ij}  \eta D_\alpha \eta v_\epsilon^i v_\epsilon^j .
\end{multline*}
By using the assumption \eqref{cond_b2}, we control the last term
\begin{align}
						\nonumber
\int_{\bR^n}  B^\alpha_{ij}  \eta D_\alpha \eta v_\epsilon^i v_\epsilon^j
&= -\int_{\bR^n}  \Phi^{\alpha\beta}_{ij} D_\beta(\eta D_\alpha \eta v_\epsilon^i v_\epsilon^j) \\
						\label{eq11.49tu}
&\le  \norm{\Phi^{\alpha\beta}_{ij}}_{\BMO(\bR^n)} \norm{D_\beta(\eta v_\epsilon^i D_\alpha \eta v_\epsilon^j)}_{\mathcal{H}^1(\bR^n)},
\end{align}
where $\mathcal{H}^1(\bR^n)$ denotes the Hardy space.
We note that the same proof of \cite[Proposition~3.2]{QianXi16} yields the following interesting estimate:
\begin{equation}				\label{eq0206s}
\norm{D_\alpha(fg)}_{\mathcal{H}^1(\bR^n)} \le C(n)\left\{ \norm{D f}_{L^2(\bR^n)} \norm{g}_{L^2(\bR^n)} + \norm{f}_{L^2(\bR^n)} \norm{Dg}_{L^2(\bR^n)} \right\}.
\end{equation}
Fix a smooth function $\tilde\eta: \bR^{n+1} \to \bR_{+}$ such that
\[
0\le \tilde\eta \le 1, \quad \tilde\eta \equiv 1\;\text{ on }\; Q_{R}(Y),\quad \tilde\eta \equiv 0\; \text{ on }\;Q_{2R}(Y)^c,\quad  \abs{D\tilde\eta} \le \tfrac{2}{R}.
\]
Since $\eta v_\epsilon^i D_\alpha \eta v_\epsilon^j= \tilde\eta \eta v_\epsilon^i D_\alpha \eta v_\epsilon^j$, by using \eqref{eq0206s}, we estimate 
\begin{align*}
\norm{D_\beta(\eta v_\epsilon^i D_\alpha \eta v_\epsilon^j)}_{\mathcal{H}^1(\bR^n)}  \le C(n) & \left\{ \left(\norm{D(\tilde\eta \eta) v_\epsilon^i}_{L^2(\bR^n)} +\norm{\tilde\eta \eta D v_\epsilon^i}_{L^2(\bR^n)}\right)\norm{D_\alpha\eta v_\epsilon^j}_{L^2(\bR^n)} \right.\\
&\quad\left. +\norm{\tilde\eta \eta v_\epsilon^i}_{L^2(\bR^n)} \left(\norm{D D_\alpha\eta v_\epsilon^j}_{L^2(\bR^n)} + \norm{D_\alpha\eta D v_\epsilon^j}_{L^2(\bR^n)}\right) \right\}.
\end{align*}
Note that $\abs{D(\tilde\eta \eta)} \le \frac{4}{R-\rho}$. Therefore, we have
\begin{multline*}
\int_{\bR^n} B^\alpha_{ij} \eta D_\alpha \eta v_\epsilon^i v_\epsilon^j
\le C(n) \Theta \left\{ \frac{1}{(R-\rho)^2} \int_{Q_{2R}\setminus Q_\rho} \abs{\vec v_\epsilon}^2 \right. \\
\left. + \frac{1}{R-\rho}  \left(\int_{\bR^n} \eta^2\abs{D \vec v_\epsilon}^2\right)^{\frac12}\left(\int_{Q_R\setminus Q_\rho} \abs{\vec v_\epsilon}^2\right)^{\frac12}+ \frac{1}{R-\rho} \left(\int_{Q_{2R}\setminus Q_\rho} \abs{\vec v_\epsilon}^2\right)^{\frac12}\left(\int_{Q_{R}\setminus Q_\rho} \abs{D\vec v_\epsilon}^2\right)^{\frac12}\right\}.
\end{multline*}
Combining together and using Young's inequality, we get
\[
\frac12\int_{\bR^n}  (\eta^2\abs{\vec v_\epsilon}^2)_t + \frac{\lambda}{2} \int_{\bR^n} \eta^2 \abs{D \vec v_\epsilon}^2
\le \frac{C}{(R-\rho)^2} \int_{Q_{2R}\setminus Q_\rho} \abs{\vec v_\epsilon}^2 \\
+ \frac{\lambda}{4} \int_{Q_R \setminus Q_\rho }\abs{D\vec v_\epsilon}^2,
\]
where $C=C(n, m, \lambda, \Lambda, \Theta)$.
Then, by integrating with respect to $t$, we obtain
\begin{equation}					\label{eq0909sun}
\sup_{t \in \bR} \int_{\bR^n} \eta^2\abs{\vec v_\epsilon}^2 + \lambda \int_{\bR^{n+1}} \eta^2 \abs{D\vec v_\epsilon}^2
\le \frac{C}{(R-\rho)^2} \int_{Q_{2R}\setminus Q_\rho} \abs{\vec v_\epsilon}^2 + \frac{\lambda}{2} \int_{Q_R \setminus Q_\rho} \abs{D\vec v_\epsilon}^2.
\end{equation}
In particular, \eqref{eq0909sun} implies
\[
\int_{Q_R(Y)^c} \abs{D\vec v_\epsilon}^2 \le \frac{C}{(R-\rho)^2} \int_{Q_{2R}\setminus Q_\rho(Y)} \abs{\vec v_\epsilon}^2 + \frac12 \int_{Q_\rho(Y)^c} \abs{D\vec v_\epsilon}^2.
\]
Since the above inequality is true for all $\rho$ and $R$ satisfying $\epsilon < \rho <R$, a well-known iteration argument yields (see \cite[Lemma~5.1]{Giaq93}) that for any $r>\epsilon$ we have
\[
\int_{Q_{2r}(Y)^c} \abs{D\vec v_\epsilon}^2 \le C r^{-2} \int_{Q_{4r}(Y)\setminus Q_r(Y)} \abs{\vec v_\epsilon}^2.
\]
Then, by setting $\rho=2r$ and $R=4r$ in \eqref{eq_eta}, we get from \eqref{eq0909sun} that
\[
\tri{\vec v_\epsilon}_{\bR^{n+1}\setminus Q_{4r}(Y)}^2 \le  Cr^{-2} \int_{Q_{8r}(Y)\setminus Q_{2r}(Y)} \abs{\vec v_\epsilon}^2 + C \int_{Q_{2r}(Y)^c} \abs{D \vec v_\epsilon}^2 \le Cr^{-2} \int_{Q_{8r}(Y)\setminus Q_{r}(Y)} \abs{\vec v_\epsilon}^2.
\]
Therefore, by Lemma~\ref{lem3.07}, we see that if $r \ge 3\epsilon$, then
\[
\tri{\vec v_\epsilon}_{\bR^{n+1}\setminus Q_{4r}(Y)}^2
\le C r^{-2} \int_{\set{ r<\abs{X-Y}_p<8r}} \abs{X-Y}_p^{-2n}\,dX  \le Cr^{-n}.
\]
We have thus shown that if $R \ge 12\epsilon$, then we have
\[
\tri{\vec v_\epsilon}_{\bR^{n+1}\setminus Q_{R}(Y)} \le C R^{-\frac{n}{2}}.
\]
On the other hand, if $R < 12\epsilon$, then by \eqref{eq17.59tu}, we have
\[
\tri{\vec v_\epsilon}_{\bR^{n+1}\setminus Q_{R}(Y)} \le \tri{\vec v_\epsilon}_{\bR^{n+1}} \le C \epsilon^{-\frac{n}{2}} \le  C R^{-\frac{n}{2}}.
\]
Therefore, we have
\begin{equation}			\label{eq11.29sun}
\tri{\vec \Gamma^\epsilon(\cdot, Y)}_{\bR^{n+1}\setminus Q_{R}(Y)} \le C R^{-\frac{n}{2}},\quad \forall \epsilon >0.
\end{equation}
In fact, by the same reasoning, we also get
\begin{equation}			\label{eq21.11tu}
\tri{\eta \vec \Gamma^\epsilon(\cdot, Y)}_{\bR^{n+1}} \le CR^{-\frac{n}{2}},\quad \forall \epsilon >0,
\end{equation}
where $\eta$ satisfies \eqref{eq_eta} with $\rho=\frac12 R$.

With the above two estimates \eqref{eq11.29sun} and \eqref{eq21.11tu} at hand, we repeat the same arguments in \cite{CDK08} and construct the fundamental solution $\vec \Gamma(X, Y)$.
By following the same proof of \cite[Theorem~2.7]{CDK08}, it is routine to verify that $\vec \Gamma(X, Y)$ satisfies all the properties stated in the theorem except the Gaussian bound \eqref{gaussian}.
\subsection{Proof of the Gaussian bound \eqref{gaussian}}
We again modify the argument in \cite{CDK08}, which is an adaptation of a method by E. B. Davies \cite{Davies}.
Let $\psi :\bR^n \to \bR $ be a bounded $C^2$ function  satisfying
\begin{equation}			\label{eq_psi}
\abs{D \psi} \le \gamma,\quad \abs{D^2 \psi} \le \delta,
\end{equation}
where $\gamma>0$ and $\delta \ge 0$ are constants to be chosen later.
For $t>s$, we define an operator $P^\psi_{s\to t}$ on $L^2(\bR^n)^m$ as follows.
For a given $\vec f\in L^2(\bR^n)^m$, let $\vec u$ be the weak solution
in $\rV^{1,0}_2((s,\infty)\times\bR^n)^N$ of the problem
\[
\left\{\begin{array}{l l}
\sL \vec u=0\\
\vec u(s,\cdot)= e^{-\psi}\vec f.\end{array}\right.
\]
Then, we define $P^\psi_{s\to t}\vec f(x):= e^{\psi(x)}\vec u(t,x)$ so that we have
\begin{equation}  \label{eq3.60.3}
P^\psi_{s\to t}\vec f(x)= e^{\psi(x)}\int_{\bR^n} \vec \Gamma (t,x,s,y)e^{-\psi(y)}\vec f(y)\,dy.
\end{equation}
We denote
\[
I(t):=\int_{\bR^n} e^{2\psi} \abs{\vec u(t,x)}^2\,dx, \quad t\ge s.
\]
Then, by \eqref{cond_b1} and \eqref{cond_c}, we have
\begin{align*}
I'(t)&= 2 \int_{\bR^n} e^{2\psi} \vec u \cdot \vec u_t 
=-2 \int_{\bR^n} \left\{ A^{\alpha\beta}_{ij}\,D_\beta u^j D_\alpha(e^{2\psi} u^i) + e^{2\psi} B^\alpha_{ij} D_{\alpha} u^j u^i +e^{2\psi} C_{ij} u^j u^i \right\}\\
&\le -2 \int_{\bR^n}  e^{2\psi} A^{\alpha\beta}_{ij}\,D_\beta u^j D_\alpha u^i -4 \int_{\bR^n}  e^{2\psi} A^{\alpha\beta}_{ij}\,D_\beta u^j D_\alpha \psi u^i
+ \int_{\bR^n}  D_\alpha(e^{2\psi} B^\alpha_{ij}) u^i u^j \\
&\le -2\lambda \int_{\bR^n} e^{2\psi} \abs{D \vec u}^2+ 4\Lambda \gamma \int_{\bR^n}  e^{2\psi} \abs{D \vec u} \,\abs{\vec u} +2 \int_{\bR^n}  B^\alpha_{ij} e^{2\psi} D_\alpha \psi u^i u^j.
\end{align*}
Similar to \eqref{eq11.49tu}, the assumption \eqref{cond_b2} yields
\[
\int_{\bR^n}  B^\alpha_{ij} e^{2\psi} D_\alpha \psi u^i u^j
\le  \norm{\Phi^{\alpha\beta}_{ij}}_{\BMO(\bR^n)} \norm{D_\beta(e^{2\psi} D_\alpha \psi u^i u^j)}_{\mathcal{H}^1(\bR^n)},
\]
Then by using \eqref{eq0206s}, we estimate (setting $f=e^\psi u^i$ and $g=e^\psi D_\alpha \psi u^j$)
\begin{multline*}
\norm{D_\beta(e^{2\psi} D_\alpha \psi u^i u^j)}_{\mathcal{H}^1(\bR^n)} \le C(n)  \left\{ \left(\norm{e^\psi D\psi u^i}_{L^2(\bR^n)} + \norm{e^\psi Du^i}_{L^2(\bR^n)}\right)\norm{e^\psi D_\alpha\psi u^j}_{L^2(\bR^n)} \right.\\
\left. +\norm{e^\psi u^i}_{L^2(\bR^n)} \left(\norm{e^\psi D\psi D_\alpha\psi u^j}_{L^2(\bR^n)} + \norm{e^\psi D D_\alpha\psi u^j}_{L^2(\bR^n)}+ \norm{e^\psi D_\alpha\psi Du^j}_{L^2(\bR^n)}\right) \right\}.
\end{multline*}
Therefore, by \eqref{eq_psi} and \eqref{cond_b2}, we obtain
\begin{multline*}
2 \int_{\bR^n}  B^\alpha_{ij} e^{2\psi} D_\alpha \psi u^i u^j \\
\le C_0 \Theta \left\{ (2\gamma^2 + \delta) \int_{\bR^n} e^{2\psi} \abs{u}^2
+2\gamma\left(\int_{\bR^n} e^{2\psi} \abs{D\vec u}^2 \right)^{\frac12} \left(\int_{\bR^n} e^{2\psi} \abs{\vec u}^2\right)^{\frac12}\right\}.
\end{multline*}
By combining together and using H\"older's and Young's inequalities, we get
\begin{equation}					\label{diff_ineq}
I'(t) \le \left\{ \left( 4\Lambda^2/\lambda +C_0^2 \Theta^2/\lambda+2 C_0\Theta \right)\gamma^2 + C_0 \Theta \delta\right\} \int_{\bR^n}  e^{2\psi} \abs{\vec u}^2.
\end{equation}
The differential inequality \eqref{diff_ineq} and the initial condition $I(s)=\norm{\vec{f}}^2_{L^2(\bR^n)}$ yields
\[
I(t)\le e^{( 2\nu \gamma^2 + 2\mu \delta) (t-s)}\norm{\vec{f}}_{L^2(\bR^n)}^2,
\]
where we set
\[
2\nu:=4\Lambda^2/\lambda +C_0^2 \Theta^2/\lambda+2 C_0\Theta\quad\text{and}\quad 2\mu:=C_0\Theta.
\]
Since $I(t)=\norm{P^\psi_{s\to t} \vec f}_{L^2(\bR^n)}^2$ for $t>s$, we have derived
\begin{equation} \label{eq3.70}
\norm{P^\psi_{s\to t} \vec f}_{L^2(\bR^n)} \le e^{(\nu\gamma^2+\mu\delta)(t-s)}\norm{\vec f}_{L^2(\bR^n)}.
\end{equation}
By \eqref{local_bdd}, we estimate
\begin{align*}
e^{-2\psi(x)} \abs{P^\psi_{s\to t} \vec f(x)}^2 = \abs{\vec u(t,x)}^2 &\le \frac{C N_0^2}{(t-s)^{\frac{n+2}{2}}} \int_s^t \!\!\!\int_{B_{\sqrt{t-s}}(x)}\abs{\vec u(\tau,y)}^2 dy d\tau\\
&\le \frac{CN_0^2}{(t-s)^{\frac{n+2}{2}}}\int_s^t \!\!\!\int_{B_{\sqrt{t-s}}(x)}e^{-2\psi(y)}
\abs{P^\psi_{s\to\tau}\vec f(y)}^2 dy d\tau.
\end{align*}
Hence, by using \eqref{eq3.70} we find
\begin{align*}
\abs{P^\psi_{s\to t}\vec f(x)}^2&\le \frac{C}{(t-s)^{\frac{n+2}{2}}}
\int_s^t \!\!\!\int_{B_{\sqrt{t-s}}(x)}e^{2\psi(x)-2\psi(y)} \abs{P^\psi_{s\to\tau}\vec f(y)}^2 dyd\tau\\
&\le \frac{C}{(t-s)^{\frac{n+2}{2}}} \int_s^t \!\!\!\int_{B_{\sqrt{t-s}}(x)}e^{2\gamma\sqrt{t-s}}
\abs{P^\psi_{s\to\tau}\vec f(y)}^2 dy d\tau\\
&\le \frac{C}{(t-s)^{\frac{n+2}{2}}} \,e^{2\gamma\sqrt{t-s}} \int_s^t e^{2(\nu\gamma^2+\mu\delta)(\tau-s)}\norm{\vec f}_{L^2(\bR^n)}^2 d\tau\\
&\le \frac{C}{(t-s)^{\frac{n}{2}}} \,e^{2\gamma\sqrt{t-s}+2(\nu\gamma^2+\mu\delta) (t-s)}
\norm{\vec f}_{L^2(\bR^n)}^2.
\end{align*}
We have thus derived the $L^2\to L^\infty$ estimate
\[
\norm{P^\psi_{s\to t}\vec f}_{L^\infty(\bR^n)} \le  C (t-s)^{-\frac{n}{4}}\,e^{\gamma\sqrt{t-s}+(\nu\gamma^2+\mu \delta)(t-s)} \norm{\vec f}_{L^2(\bR^n)}.
\]
Then, by replicating the same argument as in \cite[\S5.1]{CDK08}, we have
\[
\norm{P^\psi_{s\to t}\vec f}_{L^\infty(\bR^n)} \le  C (t-s)^{-\frac{n}{2}}\,e^{\gamma\sqrt{2(t-s)}+(\nu\gamma^2+\mu \delta)(t-s)} \norm{\vec f}_{L^1(\bR^n)},\quad \forall \vec f \in C^\infty_c(\bR^n)^m.
\]
For fixed $x, y\in \bR^n$ with $x\neq y$, the above estimate and \eqref{eq3.60.3} imply, by duality,
\begin{equation}				\label{eq3.83}
e^{\psi(x)-\psi(y)}\abs{\vec \Gamma(t,x,s,y)} \le  C (t-s)^{-\frac{n}{2}}\,e^{\gamma\sqrt{2(t-s)}+(\nu\gamma^2+\mu \delta)(t-s)}.
\end{equation}
Fix a smooth function $\psi_0: \bR \to \bR$ satisfying
\[
\psi_0(r) = 0 \;\text{ for}\; r\le 0,\quad \psi_0(r) = \abs{x-y} \; \text{ for }\;r \ge \abs{x-y},\quad \abs{\psi_0'} \le 2, \quad \abs{\psi_0''}\le 4\abs{x-y}^{-1}.
\]
We define
\[
\psi(z):= \frac{\gamma}{2} \,\psi_0(\vec n\cdot (z-y)),\quad\text{where }\;\vec n=\frac{x-y}{\abs{x-y}}.
\]
It is clear that $\psi$ is a bounded function satisfying \eqref{eq_psi} with $\delta=\frac{4\gamma}{\abs{x-y}}$.
Also, we have $\psi(x)=\frac12 \gamma \abs{x-y}$ and $\psi(y)=0$.
Therefore \eqref{eq3.83} yields
\[
\abs{\vec \Gamma(t,x,s,y)} \le  C (t-s)^{-\frac{n}{2}} \exp\left\{\gamma\sqrt{2(t-s)}+\nu\gamma^2(t-s)+ \frac{4\mu\gamma(t-s)}{\abs{x-y}}-\frac{\gamma}{2}\, \abs{x-y} \right\}.
\]
Now, we choose $\gamma= \abs{x-y}/4\nu(t-s)$.
Then
\[
\abs{\vec \Gamma(t,x,s,y)} \le  C e^{\frac{\mu}{\nu}} (t-s)^{-\frac{n}{2}}\, \exp\left\{\frac{1}{\sqrt{8}\nu} \,\frac{\abs{x-y}}{\sqrt{t-s}}- \frac{1}{16\nu} \,\frac{\abs{x-y}^2}{t-s}\right\}.
\]
Since there exists a number $N$ such that
\[
e^{\frac{1}{\sqrt{8}\nu} r- \frac{1}{16\nu} \,r^2} \le N e^{-\frac{1}{32\nu}r^2},  \quad\forall r \ge 0,
\]
we obtain the Gaussian bound \eqref{gaussian} by taking $\kappa=\frac{1}{32} \nu$.
\qed

\begin{acknowledgment}
We thank Luis Escauriaza for bringing our attention to this problem and helpful discussions.
\end{acknowledgment}

\end{document}